\documentclass{amsart}
\usepackage{amsmath,amsfonts,amssymb}
\usepackage{float,placeins,caption,subcaption}
\usepackage{graphicx}

\newcommand{\abs}[1]{\left\lvert #1\right\rvert}
\newcommand{\cM}{\ensuremath{\mathcal{M}}}
\newcommand{\cD}{\ensuremath{\mathcal{D}}}
\newcommand{\se}{\subseteq}		
\newcommand{\sm}{\setminus}		
\newcommand{\re}{\sim_R}

\newcommand{\wre}{\sim_{c}}
\newcommand{\nto}{\ensuremath{\nrightarrow}}
\newtheorem{theorem}{Theorem}
\newtheorem{lemma}{Lemma}
\newtheorem{corollary}{Corollary}
\newtheorem{conjecture}{Conjecture}
\newtheorem{quest}{Question}

\begin{document}

\title{Ramsey equivalence of $K_n$ and $K_n+K_{n-1}$}
\author{%
Thomas F. Bloom} 
\author{
Anita Liebenau}%
\thanks{The first author was supported by the Heilbronn Institute for Mathematical Research. The second author was supported by the European Research Council under the European Union's Seventh Framework Programme (FP7/2007- 2013)/ERC grant agreement no.~259385}
\begin{abstract}
We prove that, for $n\geq 4$, the graphs $K_n$ and $K_n+K_{n-1}$ are Ramsey equivalent. That is, if $G$ is such that any red-blue colouring of its edges creates a monochromatic $K_n$ then it must also possess a monochromatic $K_n+K_{n-1}$. This resolves a conjecture of Szab\'{o}, Zumstein, and Z\"{u}rcher \cite{SzZuZu:2010}. 
\end{abstract}
\address{Thomas F. Bloom\\Heilbronn Institute for Mathematical Research\\Department of Mathematics\\University of Bristol\\
University Walk\\Bristol BS8 1TW\\United Kingdom}
\email{matfb@bristol.ac.uk}
\address{Anita Liebenau\\Department of Computer Science and DIMAP\\University of Warwick\\ Coventry CV4 7AL\\ United Kingdom}
\email{a.liebenau@warwick.ac.uk}
\maketitle

A finite graph $G$ is Ramsey for another finite graph $H$, written $G\to H$, if there is a monochromatic copy of $H$ in every two-colouring of the edges of $G$. We say that $H_1$ and $H_2$ are Ramsey equivalent, written $H_1\re H_2$, if, for any graph $G$, we have $G\to H_1$ if and only if $G\to H_2$. 

The concept of Ramsey equivalence was first introduced by Szab\'{o}, Zumstein, and Z\"{u}rcher \cite{SzZuZu:2010}. A fundamental question to ask is which graphs are Ramsey equivalent to the complete graph $K_n$. It follows from a theorem of Folkman \cite{Fo:1970} that if 
a graph $H$ is Ramsey equivalent to $K_n$, then 
$\omega(H)=n$, where $\omega(H)$ denotes the size of the largest complete subgraph of $H$. 

In a recent paper, Fox, Grinshpun, Person, Szab\'{o}, and the second author \cite{FoGrLiPeSz:2014} showed that $K_n$ is \emph{not} Ramsey equivalent to any connected graph containing $K_n$. 
Furthermore, it is easily seen that $K_n$ is not Ramsey equivalent to $K_n+K_n$\footnote{Throughout the paper, 
we denote by $H_1+H_2$ the graph composed of two vertex-disjoint copies of $H_1$ and $H_2$. 
More generally, we denote by $H_1+ tH_2$ the graph that consists of a copy of $H_1$ and $t$ pairwise vertex-disjoint copies of $H_2$.}, 
see e.g.~\cite{SzZuZu:2010}.  
It follows that if $K_n\re H$ then $H$ is of the form $K_n+H'$ where $\omega(H')<n$. 

It was shown in \cite{SzZuZu:2010} that $K_n\re K_n+ K_{n-2}$. 
However, the question of whether $K_n$ is Ramsey equivalent to $K_n+K_{n-1}$ was left open. It is easily checked that $K_3$ is not Ramsey equivalent to $K_3+K_2$, since $K_6\to K_3$ but $K_6\nto K_3+K_2$. In \cite{SzZuZu:2010} it is conjectured that this is an aberration, and that $K_n\re K_n+K_{n-1}$ for large enough $n$, a conjecture repeated in \cite{FoGrLiPeSz:2014}. In this paper, we prove this conjecture.

\begin{theorem}\label{maintheorem}
For any $n\geq 4$ 
\[K_n \re K_n+K_{n-1}.\]
\end{theorem}

It is shown in \cite{FoGrLiPeSz:2014} that this is best possible, in the sense that $K_n$ is not Ramsey equivalent to $K_n+2K_{n-1}$. 
 
Our methods are combinatorial and explicit, and the idea is the following: suppose we have a graph $G$ which is Ramsey for $K_n$, and yet has been coloured so as to avoid a monochromatic $K_n+K_{n-1}$. We will then attempt, by giving an explicit recolouring of some edges, to give a colouring which no longer possesses a monochromatic $K_n$, which contradicts the Ramsey property of $G$. 

Sadly, this is not quite possible directly, and instead we will build up our proof in stages: in each lemma, we will show that either a colouring of $G$ must have a monochromatic $K_n+K_{n-1}$, or if not we can deduce some further structural information about the colouring of $G$, which will help us in the following lemmas. Eventually, we will have accumulated enough information about our supposed counterexample so that it collapses under the weight of contradiction into non-existence, which proves Theorem~\ref{maintheorem}.

As mentioned above, the clique on six vertices is an unfortunate obstruction which prevents the Ramsey equivalence of $K_3$ and $K_3+K_2$. 
Interestingly, Bodkin and Szab\'o \cite{Bo:2015} have shown that, essentially, this is the \emph{only} such obstruction.

\begin{theorem}[\cite{Bo:2015}]\label{carly+tibor}
If $G\to K_3$ and $G\nto K_3+K_2$ then $K_6\se G$. 
\end{theorem}

In Section~\ref{sec-triangle}, we give an alternative proof of this theorem, using similar techniques to those developed for the proof of Theorem~\ref{maintheorem}.

\medskip
\noindent
{\bf Notation.} All graphs are simple and finite. As a convenient abuse of notation, we write $G$ both for a graph and for its set of vertices. We write $E(G)$ for the set of edges of $G$.

\medskip
\noindent
{\bf Structure of the paper} In Section~\ref{sec-stability} we prove a Ramsey stability lemma, crucial for the proof of Theorem~\ref{maintheorem}, but which is also of independent interest. In Section~\ref{sec-proof} we give the proof of Theorem~\ref{maintheorem}. In Section~\ref{sec-triangle} we give the proof of Theorem~\ref{carly+tibor}. Finally, we conclude by giving further discussion of Ramsey equivalence, including a discussion of some still-open conjectures in this field, and adding some more.

\section{Ramsey stability}\label{sec-stability}
We first prove a lemma which may be of independent interest; we refer to it as a Ramsey stability result, since it states that if a graph $G$ is Ramsey for the clique $K_n$ then we can remove any small number of vertices and the remaining graph will still possess a Ramsey property almost as strong as the original. 

\begin{lemma}\label{lem:big}
Let $n\geq 4$ and $G\to K_n$. Let $V\subset G$ with $2n\leq \abs{V}\leq 3n-3$ and $V_0\subset V$ be any set with $\abs{V_0}\leq 2n-2$. Finally, let $x$ and $y$ be any vertices from $V\sm V_0$. 

Then, in any colouring of the edges of $G$, there exists a monochromatic copy of $K_{n-1}$ in $G\sm V_0$, say with vertex set $W$, such that either $W\cap V=\{x\}$, or $W\cap V=\{y\}$, or $x,y\not\in W\cap V$.
\end{lemma}

\begin{proof}

\begin{figure} [h] 
 \centering
 \phantom{asdfasdfasdfaasdfasfasdfad}
 \begin{subfigure}[h]{0.8\textwidth}
 \label{fig:lem1_red}
  \includegraphics[width=1.0\textwidth]{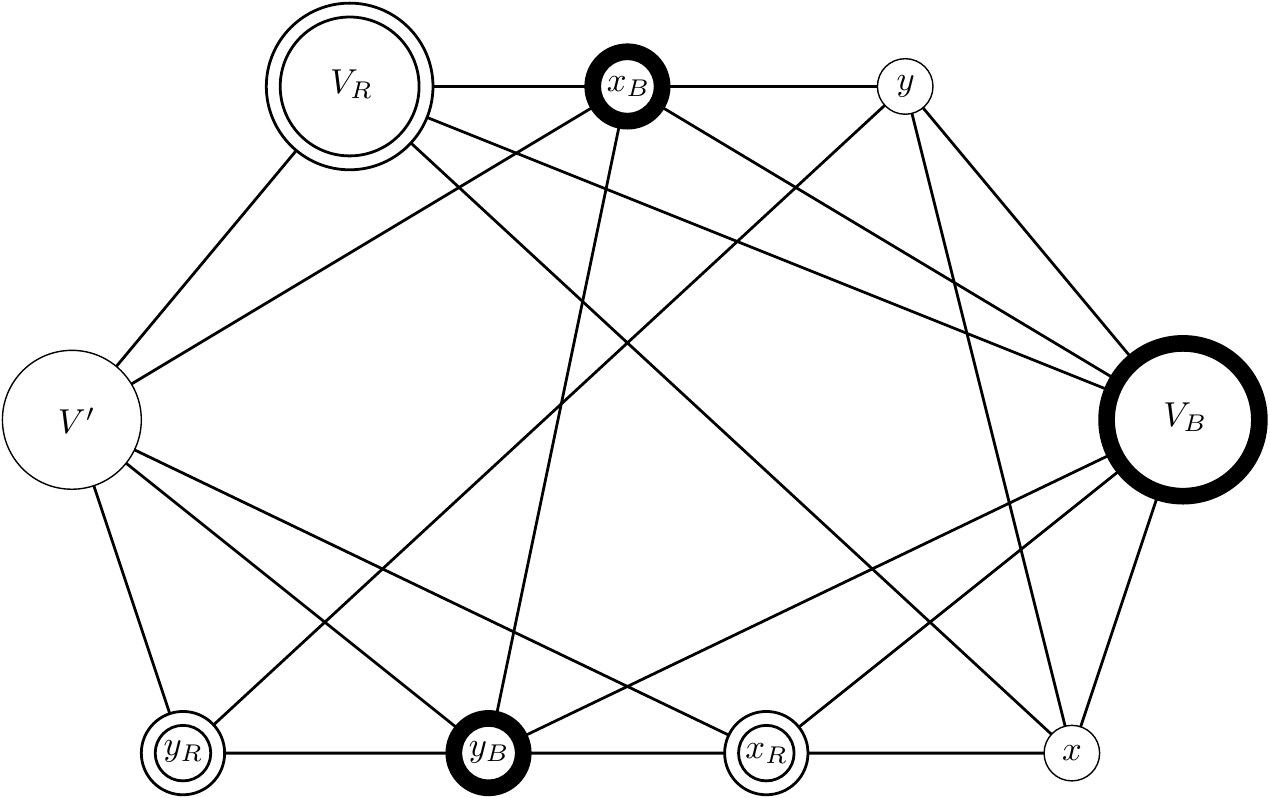}   
 \caption{The (red) edges of $G_R$}

 \end{subfigure} 
  \phantom{asdfasdfasdfaasdfasfasdfad}
 \begin{subfigure}[h]{0.8\textwidth}
  \includegraphics[width=1.0\textwidth]{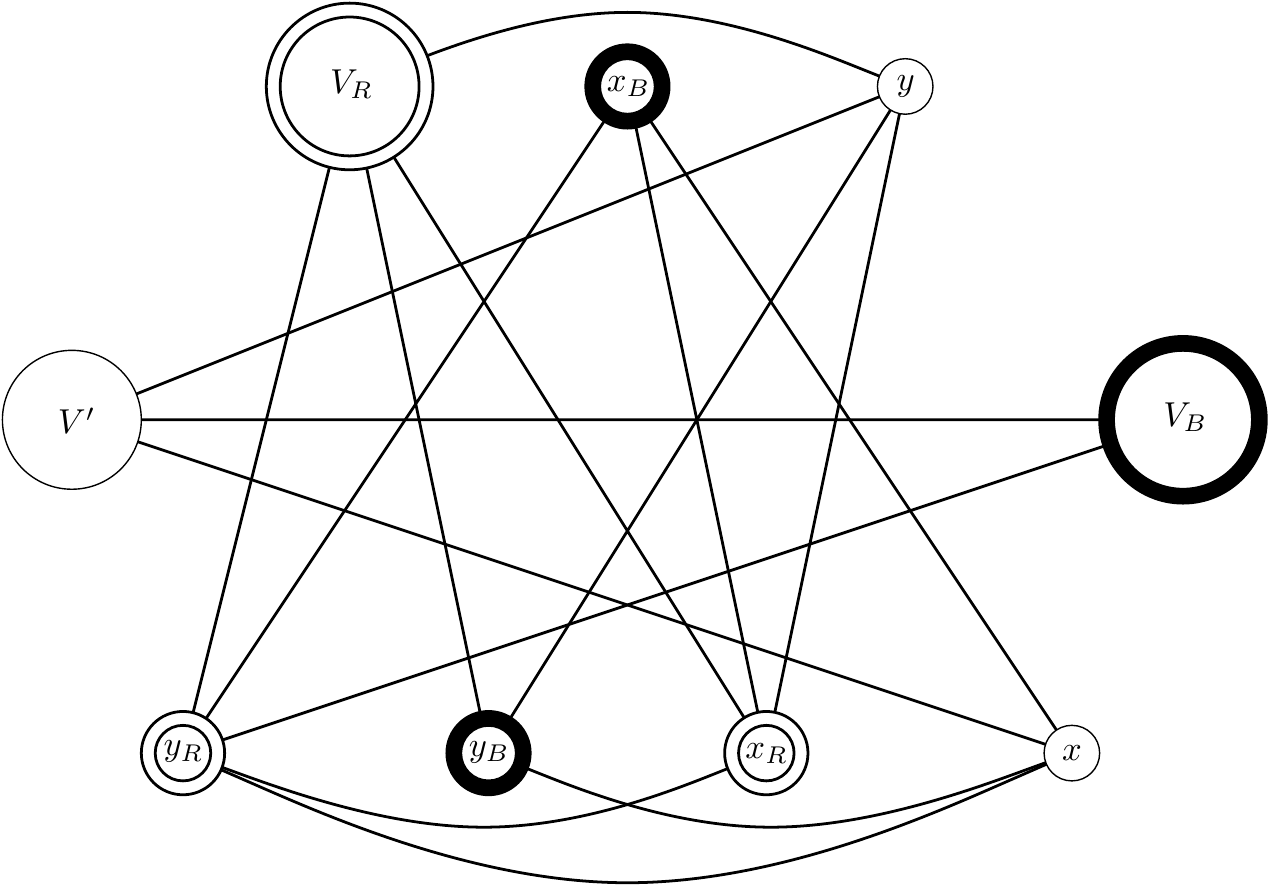}  
 \caption{The (blue) edges of $G_B$}

 \end{subfigure} 
  \caption{The recolouring of $V$ for Lemma~\ref{lem:big}. A black ring around a vertex class indicates that we colour edges between this class and $G\sm V$ blue. A white ring around a vertex class indicates that we colour edges between this class and $G\sm V$ red. No ring indicates that such edges retain their original colour. Edges inside $V_B$ are red, edges inside $V_R$ are blue, and edges inside $V'$ retain their original colour.} 
  \label{fig:lem:big}
\end{figure}
Without loss of generality, we may suppose that $\abs{V}=3n-3$ and $\abs{V_0}=2n-2$. We arbitrarily divide $V_0$ into two sets of $n-3$ vertices each, say $V_R$ and $V_B$, and four single vertices, $x_R,y_R,x_B,y_B$. For brevity, we let $V'=V\sm(V_0\cup\{x,y\})$. 
To define a recolouring of the edges incident to $V$, 
let us define an auxiliary graph $G_R$ with vertex set 
\[V(G_R)=\left\{ V_R, V_B, V',\{x\}, \{y\}, \{x_R\}, \{y_R\}, \{x_B\}, \{y_B\}\right\}.\]
Instead of giving an incomprehensible list of edges, we refer the reader 
to Figure~\ref{fig:lem:big} (A) for the definition of $G_R$. 
Let $G_B$ be the complement of $G_R$, depicted in Figure~\ref{fig:lem:big} (B). 
We now recolour the edges incident to $V$ as follows. 
If $u_1\in U_1\in G_R$ and $u_2\in U_2\in G_R$ such that $U_1\neq U_2$, 
then colour the edge $u_1u_2$ red if $U_1U_2\in E(G_R)$, 
and colour the edge $u_1u_2$ blue otherwise. 
Furthermore, colour all edges in $E(V_B)$ red, and all edges in $E(V_R)$ blue. 
The edges in $E(V')$ retain their original colouring. 
For all $u\in \{x_B,y_B\}\cup V_B$ and all $v\in G\sm V$, colour the edge $uv$ blue. For all $u\in \{x_R,y_R\}\cup V_R$ and all $v\in G\sm V$, colour the edge $uv$ red. 
It will be convenient to call the vertices in $\{x_B,y_B\}\cup V_B$ {\em blue vertices}, and to call the vertices in $\{x_R,y_R\}\cup V_R$ {\em red vertices}. The recolouring is indicated in Figure~\ref{fig:lem:big}. 

The crucial properties of this recolouring are the following, which are easy to verify from examining Figure~\ref{fig:lem:big}: 
\begin{enumerate}
\item Both $G_R$ and $G_B$ are $K_4$-free. 
\item Every triangle in $G_R$ contains at least one of $\{x_B\},\{y_B\},V_B$,  and every triangle in $G_B$ contains at least one of $\{x_R\}, \{y_R\}, V_R$. 
\item The blue vertices $V_B\cup\{x_B,y_B\}$ are connected by only red edges in $G$, and	the red vertices $V_R\cup\{x_R,y_R\}$ are connected by only blue edges in $G$. 
\end{enumerate}

Since $G$ is Ramsey for $K_n$ there must be a monochromatic copy of $K_n$ present in $G$ after this recolouring. We claim that, thanks to the fortuitous properties of our recolouring, this forces a monochromatic $K_{n-1}$ in the original colouring with the required properties. 

Let $U$ be the vertex set of the monochromatic $K_n$ present in $G$ after this recolouring. If $\abs{U\cap (V_0\cup\{x,y\})}\leq 1$ then the lemma follows immediately, since discarding at most one vertex would leave a monochromatic $K_{n-1}$ in the original colouring (as the only edges which are recoloured are incident with $V_0\cup\{x,y\}$), completely disjoint from $V_0\cup\{x,y\}$ as required.

We may suppose, therefore, that $\abs{U\cap (V_0\cup\{x,y\})}\geq 2$. The first case to consider is when $U\subset V$. By Property $(1)$ and since $n\geq 4$, $U$ must contain at least two vertices from one of the classes $V_R$, $V_B$, or $V'$. 

Suppose first that $|U\cap V_B|\geq 2$. Then $U$ must form a red $K_n$, and hence can contain at most one vertex from $V_R$, and no vertex from $V'$, since all edges between $V_B$ and $V'$ are blue. By similar reasoning, if $|U\cap V_R|\geq 2$, then $U$ must form a blue $K_n$, and hence cannot use any vertex from $V'\cup V_B$. Therefore, there exists at most one class $V''\in \{V_R,V_B,V'\}$ such that $|U\cap V''|\geq 2$. Since each such class contains at most $n-3$ vertices, discarding all but one vertex of $V''$ would force a monochromatic copy of $K_4$ within $V$, using at most one vertex from each of $V_R$, $V_B$, and $V'$. This would force a copy of $K_4$ in either $G_R$ or $G_B$, which contradicts Property $(1)$. 

Assume now that $U \not\subset V$, and suppose that $U$ hosts a red copy of $K_n$. Since all blue vertices are connected to $G\sm V$ by blue edges, $U$ cannot contain any blue vertices. Therefore, by Property $(2)$, $U$ uses vertices of at most two nodes in $G_R$. Furthermore, since the copy is red,  $|U\cap V_R|\leq1$. 

If $V'\cap U \neq \emptyset$ then $U$ can use at most one vertex from $V\sm V'$, and discarding this vertex leaves a monochromatic $K_{n-1}$ in the original colouring, completely disjoint from $V_0\cup\{x,y\}$, as required. 

If $V'\cap U = \emptyset$ then, by Property $(2)$ again, it must use exactly two vertices from $V_R\cup \{x_R,y_R,x,y\}$. Since there are only blue edges between vertices in $V_R\cup \{x_R,y_R\}$, by Property $(3)$, at least one of these two vertices in $U\cap V$ must be $x$ or $y$. Discarding the other vertex in $U\cap V$ leaves a 
monochromatic copy of $K_{n-1}$ in the original colouring which intersects $V$ in either $x$ or $y$, but no other vertices, as required. 

The case when $U$ hosts a blue copy of $K_n$ is handled similarly, and the proof is complete.
\end{proof}

The following corollary is immediate when $n\geq 4$; in this weakened form it also holds for $n=3$. We will not need Corollary~\ref{cor:small} in the rest of this paper, but it is a pleasingly simple way to demonstrate the Ramsey stability ethos of Lemma~\ref{lem:big}. 

\begin{corollary}\label{cor:small}
Let $n\geq 3$ and $G\to K_n$. If $V\subset G$ has $\abs{V}\leq 2n-2$ then $G\sm V\to K_{n-1}$. 
\end{corollary}
\begin{proof}
For $n\geq 4$ this follows immediately from Lemma~\ref{lem:big}, after expanding $V$ by two arbitrary vertices from $G\sm V$. For $n=3$, it suffices to give an explicit colouring of $K_4$ in a similar fashion, as we do in Figure~\ref{}. 

\begin{figure} [h]
 \centering
\phantom{asdfasdfasdfaasdfasfasdfad}
 \includegraphics[width=0.3\textwidth]{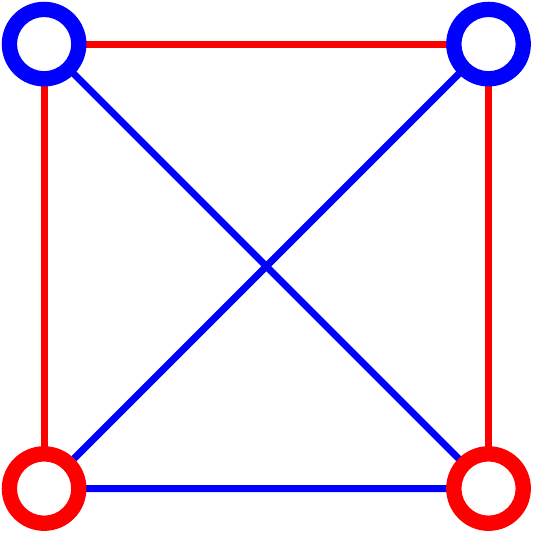}
 \caption{}
 \label{}
 \end{figure} 
 
Thus, if we recolour the edges adjacent to $V$ as indicated in Figure~\ref{}, then any monochromatic $K_3$ in $G$ must have at least two vertices from $G\sm V$, and hence $G\sm V\to K_2$ as required.
\end{proof}
\FloatBarrier

\section{Proof of the main result}\label{sec-proof}

We recall our goal: to show that $K_n$ is Ramsey equivalent to $K_n+K_{n-1}$ for $n\geq 4$. It is, of course, trivial that if $G\to K_n+K_{n-1}$ then $G\to K_n$. It remains to show that if $G\to K_n$ then $G\to K_n+K_{n-1}$. Our strategy will be to accumulate more and more information about the monochromatic structures present in a colouring of a graph, Ramsey for $K_n$, without a monochromatic $K_n+K_{n-1}$, until we are eventually able to obtain a contradiction.

\begin{lemma}\label{lem-redandblue}
Let $n\geq4$. If $G\to K_n$ then, in every colouring of $G$, there is either a monochromatic $K_n+K_{n-1}$, or a red $K_n$ and a blue $K_n$.
\end{lemma}
\begin{proof}
Suppose, without loss of generality, that the edges of $G$ are coloured so that there is a red copy of $K_n$. Let $V_R$ be the vertex set of this red $K_n$. As in the proof of Lemma~\ref{lem:big}, we will recolour some edges of $G$ and use the assumption that $G\to K_n$ to prove the claim. 

Suppose first that there is an edge $ab$ of $V_R$ which has the property that every red $K_n$ 
intersects $V_R$ in at least one vertex besides $a$ and $b$. 
In this case, we recolour every other edge of $V_R$ blue, and colour the edges between 
$V_R\sm\{a,b\}$ and $G\sm V_R$ red. 

Since $G\to K_n$ there must be a monochromatic $K_n$ in this recoloured $G$. 
Suppose first that there is a red $K_n$. If it uses at least $n-1$ vertices from $G\sm V_R$ 
then there is a red $K_{n-1}$ present in $G\sm V_R$ in the original colouring, 
and hence a red $K_n+K_{n-1}$. 
Otherwise, it must use a red edge from $V_R$. 
But the only red edge remaining in $V_R$ is $ab$, and the edges from $\{a,b\}$ to $G\sm V_R$ 
retained their original colouring. Therefore, we must have a red $K_n$ in the original colouring 
that intersects $V_R$ in exactly $\{a,b\}$, which contradicts our choice of $ab$. 
Secondly, suppose that there is a blue $K_n$ in the recoloured $G$. 
If it uses any of the new blue edges inside $V_R$, then it must be contained entirely inside $V_R$, 
since the edges from $V_R\sm\{a,b\}$ to $G\sm V_R$ are all red. 
However, this is impossible, since $V_R$ has $ab$ still coloured red. 
Therefore we must have a blue $K_n$ that uses only edges which were originally blue, and so we have a red $K_n$ and a blue $K_n$, as required. 

\begin{figure} [htbp]
 \centering
\phantom{asdfasdfasdfaasdfasfasdfad}
  \includegraphics[width=0.9\textwidth]{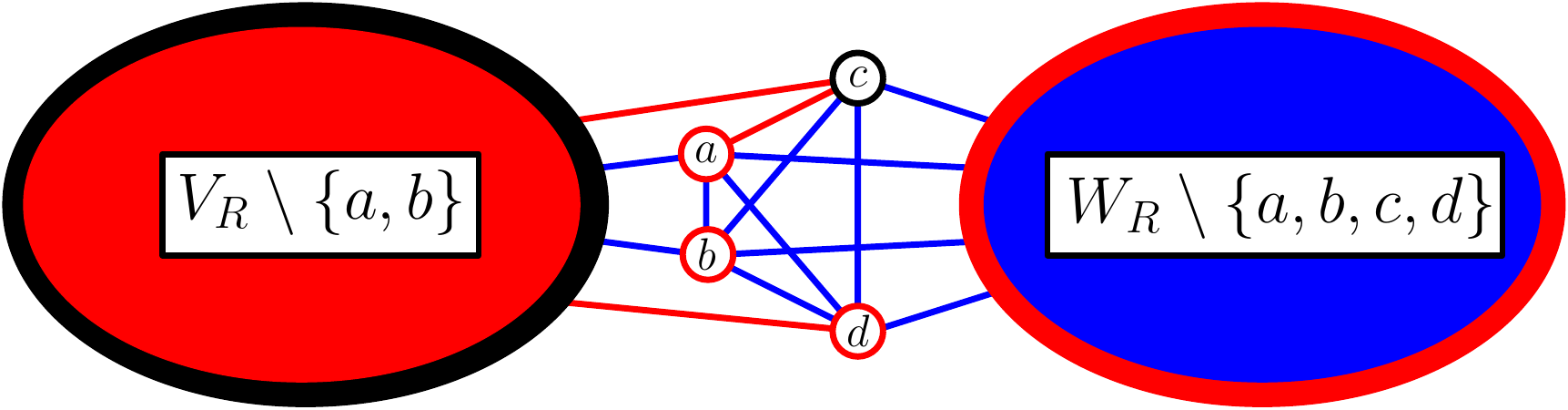}
 \caption{The colouring for Lemma~\ref{lem-redandblue}.}
 \label{fig:lemma3}
 \end{figure}

We may now assume that, for every pair $\{a,b\}\se V_R$, there is another red $K_n$ intersecting $V_R$ in only the edge $ab$. Let $W_R$ be the vertex set of another red $K_n$ such that $\abs{V_R\cap W_R}=2$, say $V_R\cap W_R=\{a,b\}$, and let $c,d$ be any two vertices in $W_R\sm V_R$. We recolour (some of) the edges incident to $W_R$ in the following way. An illustration of this colouring can be found in Figure \ref{fig:lemma3}. 
\begin{itemize}
\item 
	For all $w\in W_R\sm\{a,b,c,d\}$, all $w'\in W_R$ ($w'\neq w$), and all $v\in G\sm W_R$, 
	we colour the edge $ww'$ blue and the edge $wv$ red (if present in $G$). 
\item For all $v\in V_R\sm \{a,b\}$, we recolour the edges $av$ and $bv$ blue, 
	and the edges $cv$ and $dv$ red (if present in $G$). 
\item For all $x\in G\sm(V_R\cup W_R)$, we colour the edges $ax$, $bx$ and $dx$ in red 
	(the edge $cx$ retains its original colour). 
\item Every edge in $\{a,b,c,d\}$ is recoloured blue, except for $ac$ which remains red. 
\end{itemize}
Again, since $G\to K_n$ there must be a monochromatic $K_n$ in this recoloured $G$. 

Suppose first that there is a red $K_n$, say on vertex set $W$. 
If it uses at least $n-1$ vertices from $G\sm W_R$ 
then there is a red $K_{n-1}$ present in $G\sm W_R$ in the original colouring, 
and hence a red $K_n+K_{n-1}$. 
Otherwise, it must use a red edge from $W_R$. 
But the only red edge remaining in $W_R$ is $ac$. 
Then $W$ must be disjoint from $V_R\sm\{a\}$, since each such $ax$ is blue. 
Hence, $W\cap (V_R\cup W_R) = \{a,c\}$. But none of the edges inside $W\sm\{a\}$ were recoloured, and hence $W\sm\{a\}$ hosts a red $K_{n-1}$ in the original colouring that is vertex disjoint from $V_R$. 

Secondly, suppose that there is a blue $K_n$ in the recoloured $G$, say on vertex set $W$. 
If it uses any of the new blue edges inside $V_R\cup W_R$, then it must be contained entirely inside $V_R\cup W_R$, since the edges from $W_R\sm\{c\}$ to $G\sm (V_R\cup W_R)$ are all red. However, $V_R\cup W_R$ does not host a blue $K_n$ in this recolouring. Therefore we must have a blue $K_n$ that uses only edges which were originally blue, and so we have a red $K_n$ and a blue $K_n$, as required. 
\end{proof}

\begin{lemma}\label{lem-nand1}
Let $n\geq 4$. If $G\to K_n$ then, in any colouring of $G$, if there is a monochromatic $K_{n+1}$ then there is a monochromatic $K_n+K_{n-1}$.
\end{lemma}
\begin{proof}
Suppose that $G$ has, say, a red $K_{n+1}$, on vertex set $V_R$. By Lemma~\ref{lem-redandblue}, we may assume that there exists a blue $K_n$, say on vertex set $V_B$. Let $V=V_R\cup V_B$, so that $\abs{V}\leq 2n+1$. We now apply Lemma~\ref{lem:big}, with $V_0\subset V$ being any set of $2n-2$ vertices containing $V_B$. This yields a monochromatic $K_{n-1}$ which intersects the red $K_{n+1}$ in at most one vertex, and the blue $K_n$ not at all, and hence we must have a monochromatic $K_n+K_{n-1}$.
\end{proof}

\begin{lemma}\label{lem-2intersection}
Let $n\geq 4$, and let $G$ be a graph such that $G\to K_n$. 
Assume that there is a colouring of the edges of $G$ 
with no monochromatic copy of $K_n+K_{n-1}$. Then, in this colouring, no two  
monochromatic copies of $K_n$ intersect in exactly two vertices. 
\end{lemma}
\begin{proof}
Suppose otherwise; without loss of generality, we have two red copies of $K_n$, 
say on vertex sets $V_R$ and $V_R'$, such that $\abs{V_R\cap V_R'}=2$. 
By Lemma~\ref{lem-redandblue} we may further assume that there is a blue $K_n$, 
say on vertex set $V_B$. 

Assume first that $V_B\cap V_R \neq \emptyset$. 
Let $x\in V_R\sm (V_B\cup V_R')$ and $y\in  V_R'\sm (V_B\cup V_R)$ (which exist since $n\geq 4$ and since $V_B$ intersects with $V_R$ and $V_R'$ with at most one vertex each). 
Further, set $V:= V_R\cup V_R' \cup V_B$ and $V_0:= (V_R\cup V_B)\sm\{x\} \se V$. 
By assumption, $\abs{V}\leq 3n-3$ and $\abs{V_0}\leq 2n-2$. 
Therefore, by Lemma \ref{lem:big}, there is a monochromatic copy of $K_{n-1}$, 
say on set $W$, such that either $W\cap V = \{x\}$, 
or $W\cap V \se V\sm(V_0\cup\{x\})$. 
In the first case, when $W\cap V = \{x\}$, then $W$ is disjoint from both $V_B$ and $V_R'$, 
and hence there is a monochromatic copy of $K_n+K_{n-1}$, a contradiction. 
Otherwise, 
$W$ is disjoint from  both $V_B$ and $V_R$, 
and again, we find a monochromatic copy of $K_n+ K_{n-1}$, a contradiction. 

We argue similarly if $V_B\cap V_R' \neq \emptyset$, and therefore assume from now on that $V_B\cap (V_R\cup V_R') = \emptyset$. Let $x,y\in V_R\sm V_R'$ and $z\in V_R'\sm V_R$ be some arbitrarily chosen vertices. We again apply Lemma~\ref{lem:big}, with $V:=V_B\cup V_R\cup W$, where $W=V_R'\sm(V_R\cup\{z\})$, and $V_0:= (V_R\cup V_B)\sm \{x,y\}$. 
It is clear that $\abs{V}\leq 3n-3$ and $\abs{V_0}=2n-2$, as required. 

Suppose that there is a monochromatic copy of $K_{n-1}$ which intersects $V$ 
in only vertices of $W$. In particular, it is vertex-disjoint from $V_B\cup V_R$, 
and hence it creates a monochromatic $K_n+K_{n-1}$, which is a contradiction. 

It follows that there exists a monochromatic copy of $K_{n-1}$ which intersects $V$ 
in either $x$ or $y$, but no other vertices. 
Since it is disjoint from $V_B$, we may assume that it is red. 
If this red $K_{n-1}$ does not use $z$, however, 
then together with $V_R'$ we have a red $K_n+ K_{n-1}$, which is a contradiction. 
Therefore, either $xz$ or $yz$ is red. Since $x$ and $y$ were an arbitrary choice of 
two vertices from $V_R\sm V_R'$, it follows that all but at most one vertex of $V_R$ is connected 
to $z$ by a red edge. 

That is, $V_R\cup\{z\}$ hosts two red copies of $K_n$ that intersect in $n-1$ vertices. 
Note that if $V_R\cup\{z\}$ forms in fact a red copy of $K_{n+1}$, then we are done by 
Lemma \ref{lem-nand1}. 
Therefore, to finish the argument, let $x\in V_R \sm V_R'$ such that the edge $xz$ is blue or not present in $G$. 
As noted, there is at most one such $x$. 
We apply Lemma~\ref{lem:big} yet again to reach a contradiction. 
Let $y\in V_R\sm(V_R'\cup\{x\})$, set $V_0:= (V_R\cup V_B)\sm \{x,y\}$ and 
$V:= (V_R\cup V_R'\cup V_B)\sm \{x\}$. 
Then $\abs{V}=3n-3$ and $\abs{V_0}= 2n-2$. 
By Lemma~\ref{lem:big}, there exists a monochromatic copy of $K_{n-1}$, say on vertex set $W$, 
such that either $W\cap V = \{y\}$, $W\cap V = \{z\}$, or $W\cap V \se V\sm(V_0\cup\{y,z\})$. 
If $W\cap V = \{y\}$, then $W$ is disjoint from $V_R'\cup V_B$ and hence forms a monochromatic 
copy of $K_n+K_{n-1}$ in the original colouring, a contradiction. 
If $W\cap V = \{z\}$, then $W$ is disjoint from $V_B$, and hence we may assume that it is red. 
But then, $W$ is either disjoint from $V_R$ and forms a red copy of $K_n+K_{n-1}$, 
or $x\in W$, and hence the edge $zx$ is red, a contradiction. 
Finally, if $W\cap V \se V\sm(V_0\cup\{y,z\})$, then $W$ together with $V_0\cup \{y,z\}$ forms 
a monochromatic copy of $K_n+K_{n-1}$. 
\end{proof}
We will now conclude the proof of the main result.

\begin{proof}[Proof of Theorem \ref{maintheorem}]
Let $n\geq 4$, and let $G$ be a graph such that $G\to K_n$. 
Assume that there exists a colouring of the edges of $G$ without a monochromatic copy of $K_n+K_{n-1}$. By Lemma \ref{lem-redandblue}, we can assume that there are two (not necessarily disjoint) sets $V_R$ and $V_B$ of vertices such that $G[V_R]$ and $G[V_B]$ form a red and a blue copy of $K_n$, respectively.

By assumption, any other red (blue) copy of $K_n$ intersects $V_R$ ($V_B$) in at least two vertices; in fact, by Lemma \ref{lem-2intersection}, any other red (blue) copy of $K_n$ intersects $V_R$ ($V_B$) in at least three vertices. 
That is, every set $W_R\subset V_R$ of size $\abs{W_R}=n-2$ meets every red copy of $K_n$ in at least one vertex, and every set $W_B\subset V_B$ of size $\abs{W_B}=n-2$ meets every blue copy of $K_n$ in at least one vertex. 

If $V_R\cap V_B = \emptyset$, fix two arbitrary subsets $W_R\subset V_R$ and $W_B\subset V_B$, both of size $\abs{W_R}=\abs{W_B}=n-2$. If $V_R\cap V_B \neq \emptyset$, let $W_B\se V_B$ be a set of size $n-2$ such that $V_R\cap V_B\se W_B$, 
and let $W_R\se V_R$ be a subset of size $n-3$ such that $W_R\cap V_B = \emptyset$ 
(note that $\abs{V_R\cap V_B}= 1$). 
In both cases, the sets $W_R\subset V_R$ and $W_B\subset V_B$ are disjoint and, by the above discussion, any monochromatic copy of $K_n$ meets $W_R\cup W_B$ in at least one vertex.

We now recolour the graph and show that the resulting colouring does not contain a monochromatic copy of $K_n$. We may assume, without loss of generality, that all edges in $V_R\cup V_B$ are present, since losing edges will only help prevent a monochromatic $K_n$ occurring. Let $\{x_R,y_R\}=V_R\sm (W_R\cup W_B)$ and $\{x_B,y_B\}=V_B\sm (W_R\cup W_B)$. 
\begin{itemize}
\item If $n=4$ and $V_R\cap V_B \neq \emptyset$ (i.e.~$|W_R|=1$), colour one edge between $W_R$ and $W_B$ red, and the other one blue. 
	Otherwise, colour the edges between $W_R$ and $W_B$ so that 
	for every $v\in W_R$ there are $w_r,w_b\in W_B$ such that $vw_r$ is red and $vw_b$ is blue, and for every $v\in W_B$ there are $w_r,w_b\in W_R$ such that $vw_r$ is red and $vw_b$ is blue.\footnote{This is clearly possible if $\abs{W_R}, \abs{W_B}\geq 2$, i.e.~if $V_R\cap V_B = \emptyset$ or $n\geq 5$.}	
\item For all $x\in W_R$, $y\in V_R$, and $z\not\in V_R\cup W_B$,  
	colour the edge $xy$ blue and colour the edge $xz$ red.  
\item For all $x\in W_B$, $y\in V_B$, and $z\not\in W_R\cup V_B$, 
	colour the edge $xy$ red and colour the edge $xz$ blue.  
\end{itemize}
This recolouring is illustrated in Figure~\ref{fig:proof} (where we label as black those edges which retain their original colouring). 

\begin{figure} [htbp]
 \centering
\phantom{asdfasdfasdfaasdfasfasdfad}
 \includegraphics[width=0.5\textwidth]{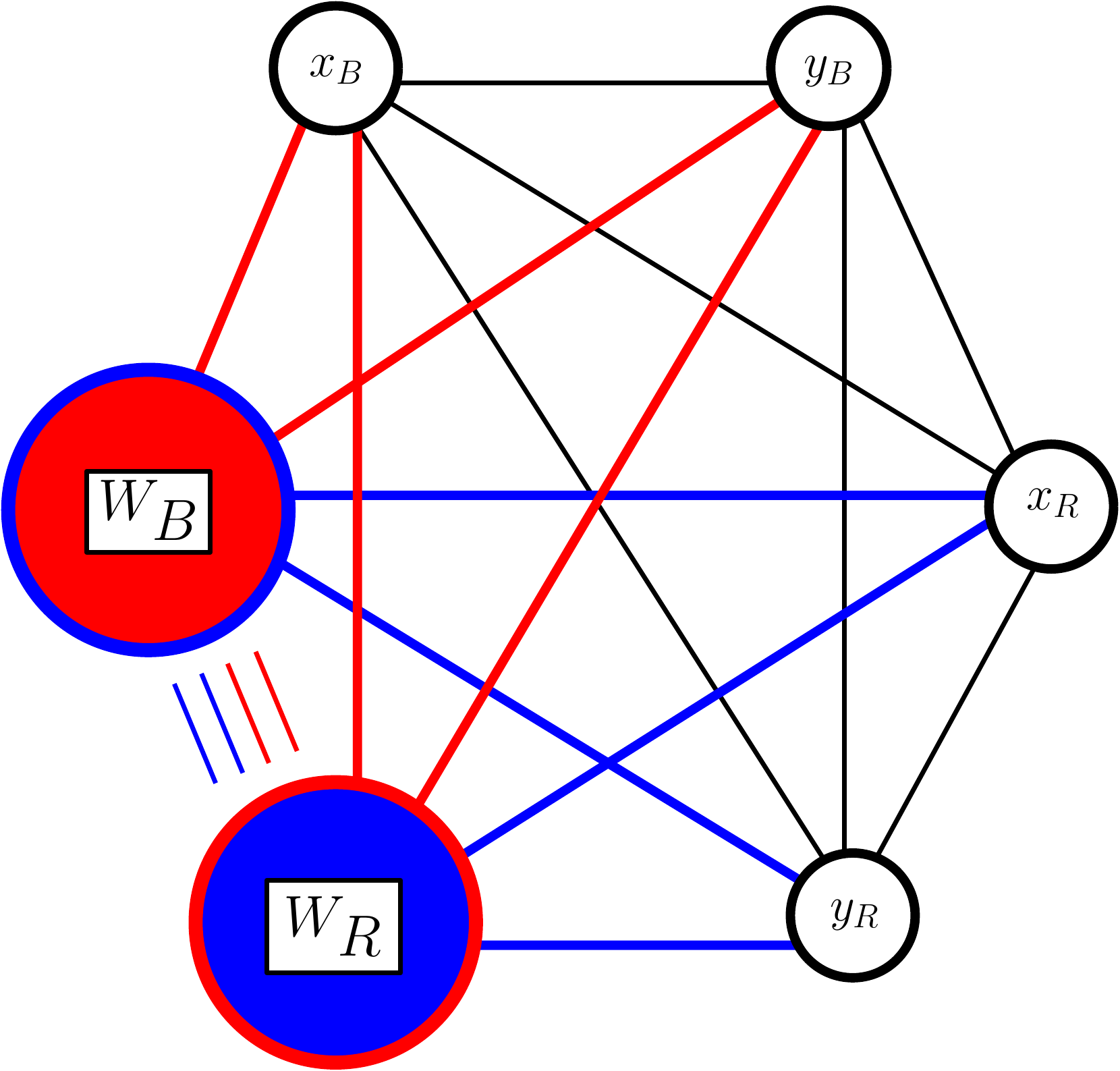}
 \caption{The colouring for the proof of Theorem~\ref{maintheorem}.}
 \label{fig:proof}
 \end{figure}

Note that we only recolour edges incident to $W_R\cup W_B$. 
Therefore, by our choice of $W_R\cup W_B$, any monochromatic 
copy of $K_n$ (after recolouring the edges) must meet $W_R\cup W_B$ 
in at least one vertex. 

Suppose now that a red $K_n$ exists and uses vertices from $W_R$ but not $W_B$. Then it must use just one vertex from $W_R$ and $n-1$ from $G\sm(V_R\cup W_B)$, and hence we have a red $K_n+K_{n-1}$ in the original colouring. If a blue $K_n$ exists and uses vertices from $W_R$ but not $W_B$, then it cannot use any vertices from $\{x_B,y_B\}$ or $G\sm (V_R\cup V_B)$, and can only use at most one vertex from $\{x_R,y_R\}$ (since the edge $x_Ry_R$ remains red as in the original colouring). 
But this contradicts the fact that $\abs{W_R}\leq n-2$. 

Similarly, we can rule out the case that a monochromatic copy of $K_n$ uses vertices from $W_B$ but not $W_R$. Therefore, if there is a monochromatic copy of $K_n$ after recolouring the edges, then it must use vertices from both $W_R$ and $W_B$. Assume first that this copy is red. Since all vertices in $W_B$ are connected to $G\sm(V_R\cup V_B)$ via blue edges, the red copy of $K_n$ must lie entirely inside $V_R\cup V_B$. But then, it can use at most one vertex from $W_R$ and at most one of $\{x_B,y_B\}$. The remaining $n-2$ vertices must come from $W_B$, so we must use all vertices from $W_B$. However, in the case $V_R\cap V_B = \emptyset$ or $n\geq 5$, 
every vertex in $W_R$ sees at least one vertex of $W_B$ in blue. 
In the case $n=4$, $\abs{W_R}=1$ and $\abs{W_B}=2$, the two edges between $W_R$ and $W_B$ are of opposite colour, 
and hence, at most one vertex of $W_B$ can contribute to a red $K_4$. 

A similar argument shows that we do not find a blue copy of $K_n$ using vertices from both $W_R$ and $W_B$. We have therefore constructed a colouring of $G$ which has no monochromatic $K_n$, contradicting the original Ramsey property of $G$ and concluding the proof.
\end{proof}

\section{Ramsey equivalence of $K_3$}\label{sec-triangle}
In this section we give a proof of Theorem~\ref{carly+tibor}, a result of Szab\'{o} and Bodkin~\cite{Bo:2015}. We need to show that, if $G\to K_3$ and $G\nrightarrow K_3+K_2$, then $K_6\subset G$.

\begin{proof}[Proof of Theorem~\ref{carly+tibor}]
Let $G$ be a graph which is Ramsey for $K_3$ and not Ramsey for $K_3+K_2$, and fix some colouring of $G$ with no monochromatic $K_3+K_2$. We first show that $G$ must possess both a red $K_3$ and a blue $K_3$. 

Without loss of generality, there is a red $K_3$, say on vertex set $V_R=\{x_R,y_R,z_R\}$. We now recolour the edges $x_Ry_R$ and $x_Rz_R$ blue, and colour all the edges from $x_R$ to $G\sm V_R$ red. 
It is now straightforward that a blue copy of $K_3$ must be a blue copy in the original colouring, and that a red copy of $K_3$ forces either a monochromatic 
copy of $K_3+K_2$ in the original colouring, or it uses the edge $y_Rz_R$ and single new vertex, say $v_R$. 
In this case, we recolour once again in the following way, as indicated in Figure~\ref{fig:k4}. 
\begin{figure} [htbp]
 \centering
\phantom{asdfasdfasdfaasdfasfasdfad}
 \includegraphics[width=0.3\textwidth]{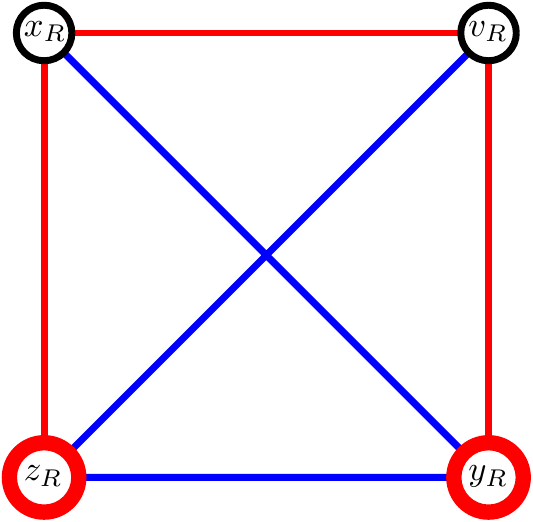}
  \caption{}
 \label{fig:k4}
  \end{figure}
  
We colour the three-edge path $(z_R,x_R,v_R,y_R)$ red, and the complement in $V_r\cup\{v_R\}$ blue. 
Furthermore, we colour all edges between $\{z_R,y_R\}$ and $G-(V_r\cup\{v_R\})$ red. 
As before, if there is now a blue $K_3$, then it cannot use either of the vertices $y_R$ or $z_R$, 
and hence it must have been already present in the original colouring of $G$. 
Otherwise, a red $K_3$ must use exactly two vertices from $\{x_R,v_R,y_R,z_R\}$.  
In particular, we have a red $K_2$ that is either disjoint from $\{x_R,y_R,z_R\}$ or $\{v_R,y_R,z_R\}$, 
and hence a red $K_3+K_2$ in the original colouring. 

We have shown that there must be, in our coloured graph $G$, a red $K_3$, say on $V_R$, 
and a blue $K_3$, say on $V_B$. 
We now show that we can assume that $V_R$ and $V_B$ are disjoint. 

Suppose that our original choices are not, so that $\abs{V_R\cup V_B}=5$. Suppose $V_R\cap V_B=\{x\}$ and $V_R=\{x,y_R,z_R\}$ and $V_B=\{x,y_B,z_B\}$. Clearly, any edges between $\{y_R,z_R\}$ and $G\sm(V_R\cup V_B)$ must be red. If their neighbourhoods intersect in $G\sm (V_R\cup V_B)$ we have found another red $K_3$, entirely disjoint from $V_B$, and we may proceed. Otherwise, we may assume that the neighbourhoods of $y_R$ and $z_R$ in $G\sm (V_R\cup V_B)$ are disjoint. 
Similarly, we can assume that the neighbourhoods of $y_B$ and $z_B$ in $G\sm (V_R\cup V_B)$ are disjoint. 
We now colour the edges incident to $V_R\cup V_B$ as indicated in Figure~\ref{fig:5v}.
\begin{figure} [htbp]
 \centering
\phantom{asdfasdfasdfaasdfasfasdfad}
 \includegraphics[width=0.5\textwidth]{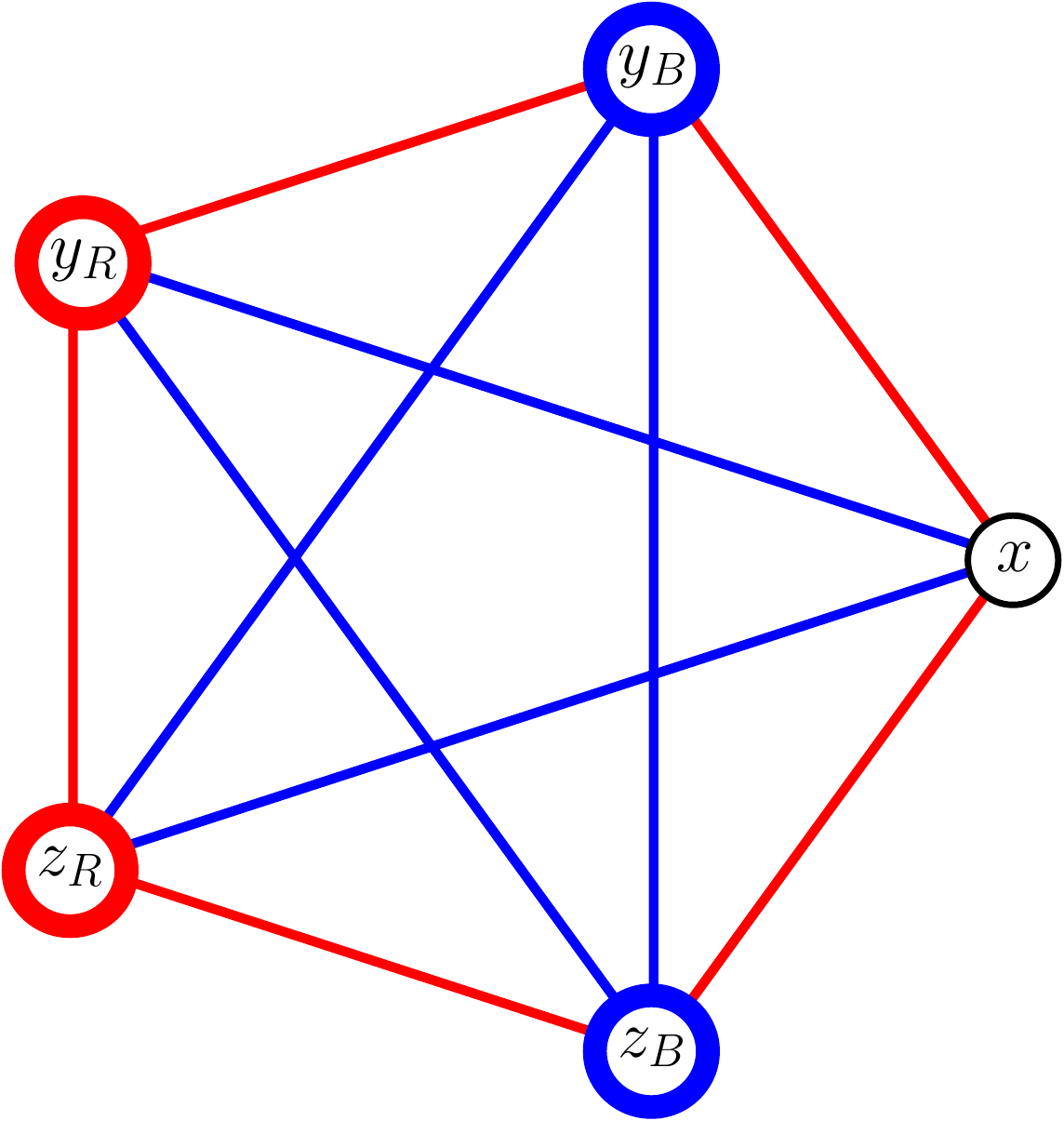}
  \caption{The recolouring when $V_R\cap V_B=\{x\}$. The edges between $\{y_R,z_R\}$ and $G-(V_R\cup V_B)$ are red, the edges between $\{y_B,z_B\}$ and $G-(V_R\cup V_B)$ are blue.}
 \label{fig:5v}
  \end{figure}
 Since $G\to K_3$, there must be a monochromatic copy of $K_3$ after recolouring. 
 Furthermore, it must intersect $V_R\cup V_B$ in exactly two vertices, since the original colouring would contain a monochromatic $K_3+K_2$ otherwise. 
If it is a red $K_3$, say, then it must therefore use $y_R$, $z_R$, and a single vertex from $G\sm (V_R\cup V_B)$, which contradicts the fact that their neighbourhoods are disjoint as discussed above, and we argue similarly if we have found a blue $K_3$. 

We may therefore assume that we have produced two disjoint sets, $V_R$ and $V_B$, each of which spans a red and blue $K_3$ respectively. 

Suppose first that there are two vertex-disjoint edges missing from $V_R\cup V_B$. We then recolour the edges incident to 
$V_R\cup V_B$ as in Figure~\ref{fig:k62missing} (where, as usual, a red (blue) vertex represents the fact that the edges between that vertex and $G\sm (V_R\cup V_B)$ are coloured red (blue)). 
\begin{figure} [htbp]
 \centering
\phantom{asdfasdfasdfaasdfasfasdfad}
 \includegraphics[width=0.5\textwidth]{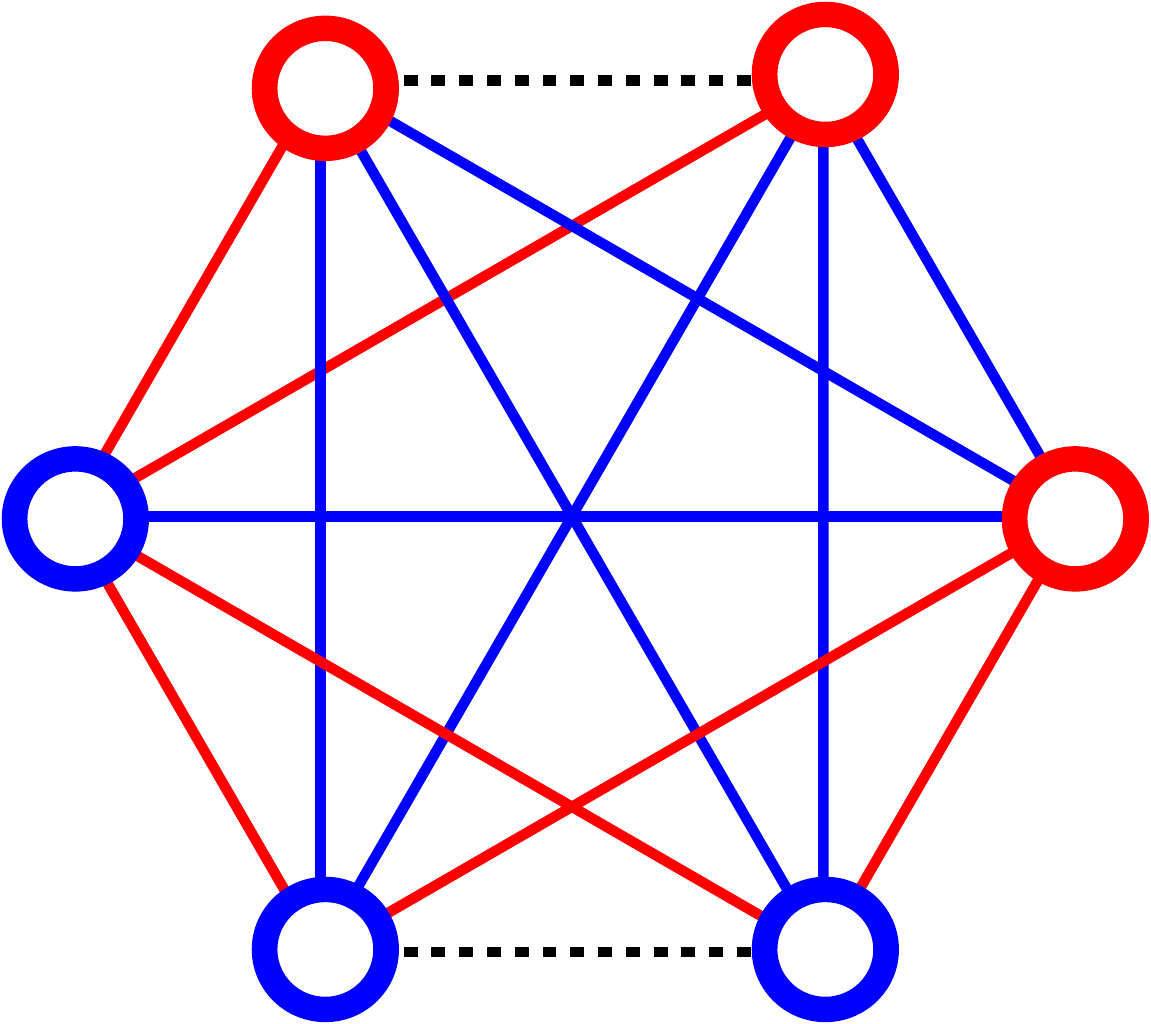}
  \caption{}
 \label{fig:k62missing}
  \end{figure}
It is easy to check that this colouring of $V_R\cup V_B$ contains no monochromatic $K_3$. 
Moreover, there are no blue edges between blue vertices, and, vice versa, no red edges between red vertices. 
It follows that a monochromatic copy of $K_3$ in this recoloured $G$ must use at least two vertices from $G\sm (V_R\cup V_B)$, which would create a monochromatic $K_3+K_2$ in the original colouring of $G$, a contradiction.

We may suppose, therefore, that there is a vertex, without loss of generality say $x_R\in V_R$, such that every missing edge in $V_R\cup V_B$ is adjacent to $x_R$. Furthermore, we may suppose that at least one edge is missing, or else we have a $K_6$ in $G$ as required. Let $x_Rx_B$ be some missing edge, where $x_B\in V_B$. 

Assume first that there is a vertex, say $w$, in $G\sm (V_R\cup V_B)$ that has at least five neighbours in $V_R\cup V_B$. 
If it is adjacent to every vertex of $(V_R\cup V_B)\sm\{x_R\}$ then this creates a $K_6$, as required. 
Hence, we can assume that $wx_R$ is an edge in $G$. 
Furthermore, all edges between $w$ and $V_R$ (if present in $G$) must be red, 
and all edges between $w$ and $V_B$ must be blue (as otherwise they create a monochromatic copy 
of $K_3+K_2$ in the original colouring). 

Suppose that $w$ is adjacent to every vertex of $V_R$ and to two vertices of $V_B$, say $a$ and $b$, and that the edge $wc$ is missing. 
If either of the edges $x_Ra$ or $x_Rb$ is missing, then by considering $\{w,x_R,y_R\}\cup V_B$ we have a similar situation as above -- namely, 
disjoint vertex sets of a red and a blue copy of $K_3$ with two vertex disjoint edges missing, and we are done. 
Otherwise, we have a $K_6$ in $\{w,x_R,y_R,z_R,a,b\}$. 
Suppose now that $w$ is adjacent to every vertex of $V_B$ and $x_R$ and some other vertex of $V_R$, say $a$, and the edge $wb$ is missing, where $b\in V_R$. 
As above, we are now done by considering $V_R\cup \{w,x_B,y_B\}$, since $wb$ and $x_Bx_R$ are two independent edges missing. 

For the remainder of the argument, we may therefore assume that 
every vertex of $G\sm (V_R\cup V_B)$ has at most four neighbours in $V_R\cup V_B$.
We now describe a recolouring of the edges incident to $V_R\cup V_B$ such that there is 
no monochromatic $K_3$ that uses at least two vertices from $V_R\cup V_B$. 
Recolour the interior edges of $V_R\cup V_B$ as in Figure~\ref{fig:k6missing}. 
\begin{figure} [htbp]
 \centering
\phantom{asdfasdfasdfaasdfasfasdfad}
 \includegraphics[width=0.5\textwidth]{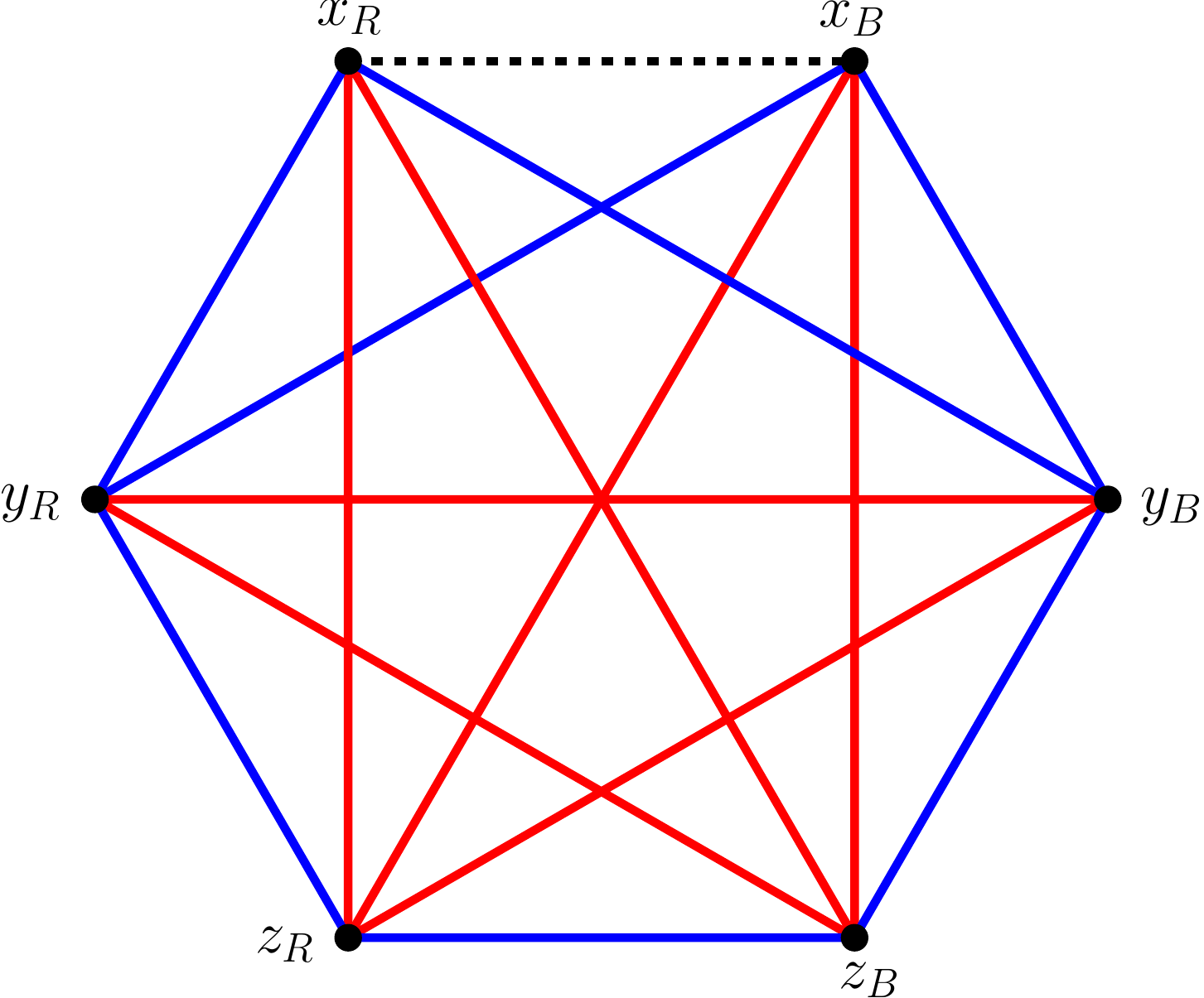}
  \caption{}
 \label{fig:k6missing}
  \end{figure}
Let now $w\in G\sm(V_R\cup V_B)$ and let $N_w\se V_R\cup V_B$ be any set of four vertices containing $N(w)\cap(V_R\cup V_B)$. 
Then, either $(1)$, $\{x_R,x_B\}\not\se N_w$ and we see a red copy of the three-edge path $P_3$ and 
a blue copy of $P_3$ in $N_w$, 
or $(2)$, $\{x_R,x_B\}\se N_w$, say $N_w=\{a,b,x_R,x_B\}$, for some $a,b$, and we 
see a monochromatic copy of $C_4$ or a monochromatic star $K_{1,3}$ with $a$ being the centre of the star. 

In case $(1)$, say $(a,b,c,d)$ forms the red $P_3$ in $N_w$, then we colour the edges $wb$ and $wc$ blue, and the edges $wa$ and $wd$ red (if present in $G$). 
In case $(2)$, we colour the edge $wa$, $wx_R$ and $wx_B$ the opposite colour of $ax_R$ and $wb$ the same colour as $ax_R$.  

Note first that the colouring of $V_R\cup V_B$ does not contain a monochromatic triangle. 
Furthermore, it is evident that we do not create a monochromatic triangle 
on vertices $w,x,y$ with $x,y\in V_R\cup V_B$ and $w\not \in V_R\cup V_B$, 
since no such $w$ sees both vertices of a red edge in red nor both vertices of a blue edge in blue.

However, since $G\to K_3$, there must be an edge $vw$ with $v,w\not\in V_R\cup V_B$, which creates a monochromatic $K_3+K_2$ 
in the original colouring, a contradiction. 
\end{proof}

\section{Further remarks}
We have shown that $K_n$ and $K_n+K_{n-1}$ are Ramsey equivalent for $n\geq 4$. Furthermore, we have seen that $K_6$ is the only obstruction to the Ramsey equivalence of $K_3$ and $K_3+K_2$, 
i.e.~any graph $G$ that satisfies $G\to K_3$ and $G\nrightarrow K_3+ K_2$ must contain $K_6$ as a subgraph. 

The only pairs of graphs $(H_1,H_2)$ known to be Ramsey equivalent are of the form $H_1\cong K_n$ and $H_2\cong K_n+H_3$, where $H_3$ is a graph of clique number less than $n$. 
Furthermore, it is known (\cite{FoGrLiPeSz:2014} and \cite{NeRo:1976}) that the only connected graph that is Ramsey equivalent to $K_n$ is the clique $K_n$ itself. 

It is an open question, first posed in \cite{FoGrLiPeSz:2014}, whether there are two connected non-isomorphic graphs $H_1$ and $H_2$ that are Ramsey equivalent. It follows from \cite{NeRo:1976} that, if such a pair exist, they must have the same clique number. In \cite{AxRoUe:2015} it is shown that they must also have the same chromatic number, under the assumption that one of the two graphs satisfies an additional property, called {\em clique-splittability}. 

To tackle problems on Ramsey equivalence, a weaker concept was proposed by Szab\'o \cite{Sz:2015}. We will first introduce some necessary notation.  
We say that $G$ is Ramsey minimal for $H$ if $G$ is Ramsey for $H$ and no proper subgraph of $G$ is Ramsey for $H$. Denote by $\mathcal{M}(H)$ the set of all graphs which are Ramsey minimal for $H$, and by $\mathcal{R}(H)$ the set all graphs which are Ramsey for $H$. Finally, let $\cD(H_1,H_2):=(\cM(H_1)\sm \mathcal{R}(H_2))\cup (\cM(H_2)\sm \mathcal{R}(H_1))$ be the class of graphs $G$ that are Ramsey minimal for $H_1$, but which are not Ramsey for $H_2$, or vice versa. Equivalently, $\cD(H_1,H_2)$ is the set of minimal obstructions to the Ramsey equivalence of $H_1$ and $H_2$.

In particular, $H_1$ and $H_2$ are Ramsey equivalent if and only if $\cD(H_1,H_2)=\emptyset$. We say that $H_1$ and $H_2$ are Ramsey close, denoted by $H_1\wre H_2$, if $\cD(H_1,H_2)$ is finite. We stress that this is not an equivalence relation: reflexivity and symmetry are trivial, but transitivity does not hold, since every graph containing at least one edge is close to $K_2$.

Two graphs may be Ramsey close in a rather trivial sense if $\cM(H_1)$ and $\cM(H_2)$ are both finite, or if $H_2\subset H_1$ and $\cM(H_2)$ is finite. Graphs such that $\cM(H)$ is finite are known as Ramsey-finite graphs. The class of Ramsey-finite graphs has been studied quite intensively; see, for example, \cite{BuErFaRiSc:1982} for some results and further references. In particular, it has been shown that the only Ramsey-finite graphs are disjoint unions of stars.

If one wishes to prove that two graphs are Ramsey equivalent, a possible first step is to show that the two graphs are Ramsey close. Szab\'o \cite{Sz:2015} has posed the following weaker version of the open problem mentioned earlier.

\begin{quest}\label{qu1}
Is there a pair of non-isomorphic, Ramsey-infinite, connected graphs which are Ramsey close?
\end{quest}

We suspect that the answer to Question \ref{qu1} is negative, even with this weakening of the notion of Ramsey equivalence.

Ne\v{s}et\v{r}il and R\"{o}dl \cite{NeRo:1978} proved that if $\omega(H)\geq 3$ then there exist infinitely many Ramsey-minimal graphs $G\in\cM(H)$ such that $\omega(H)=\omega(G)$. In particular, it follows that if $\omega(G_1)\geq 3$ and $\omega(G_2)\geq 3$, and $G_1\wre G_2$, then $\omega(G_1)=\omega(G_2)$.

Theorem~\ref{carly+tibor} states that, although $K_3$ and $K_3+K_2$ are not Ramsey equivalent, they are Ramsey close. Indeed, the only graph $G$ that is Ramsey minimal for $K_3$ and satisfies $G \nto K_3+K_2$ is $K_6$ itself. 
This is the only example of a pair of Ramsey-infinite graphs which are Ramsey close but not Ramsey equivalent that we know of. 
In this case, $\abs{\mathcal{D}(K_3, K_3+K_2)}=1$. We pose the following.

\begin{quest}
For any integer $k\geq 2$, is there a pair of Ramsey-infinite graphs $H_1$ and $H_2$ such that $\abs{\mathcal{D}(H_1,H_2)}=k$?
\end{quest}

An affirmative answer, which we believe to exist, would in particular imply the following conjecture. 

\begin{conjecture}
There are infinitely many pairs of Ramsey-infinite graphs which are Ramsey close but not Ramsey equivalent.
\end{conjecture} 

We close this paper with the following simpler question, which one may be able to answer negatively by adapting the methods of \cite{FoGrLiPeSz:2014}.

\begin{quest}
Are $K_n$ and $K_n+K_n$ Ramsey close?
\end{quest}

\end{document}